\documentclass[english]{msjproc}
\usepackage{lmodern}

\usepackage{etex}
\usepackage{amsmath,amsfonts,amssymb,amsthm,amscd,mathrsfs}
\usepackage{pst-all}
\psset{unit=1mm}
\usepackage[all]{xy}
\usepackage{tikz}
\usetikzlibrary{shapes.geometric,shadows}
\usepackage{enumitem,pifont}
\usepackage{colortbl,multirow}
\setlist{nolistsep}
\AtBeginDocument{
      \setlength\abovedisplayskip{3pt}
      \setlength\belowdisplayskip{4pt}}
\makeatletter

\newcommand{\FF}{\mathcal F_1}
\newcommand{\MRR}{\mathcal R}
\newcommand{\QQ}{\mathbb{Q}}
\newcommand{\ZZ}{\mathbb{Z}}
\newcommand{\NN}{\mathbb{N}}
\newcommand{\OO}{\mathcal O}
\newcommand{\Orb}{\operatorname{Orb}}
\newcommand\Rack{{\scriptscriptstyle\mathrm{R}}}
\newcommand\op{\mathrel{\triangleright}}

\newcommand{\End}{\operatorname{End}}

\newcommand{\Id}{\operatorname{Id}}

\renewcommand{\ge}{\geqslant}
\renewcommand{\le}{\leqslant}
\newcommand\oa{\mathbf{a}}
\newcommand\ooa{\widehat{\oa}}
\newcommand\ob{\mathbf{b}}
\newcommand\oc{\mathbf{c}}
\newcommand\od{\mathbf{d}}
\newcommand\oi{\mathbf{i}}
\newcommand\ooi{\widehat{\oi}}
\newcommand\oj{\mathbf{j}}

\newcommand\ophi{\widehat{\phi}}
\newcommand\wP{\widetilde{P}}
\newcommand{\bi}[1]{\textbf{\textit{#1}}}

\newcommand\mapsfrom{\mathrel{\reflectbox{\ensuremath{\mapsto}}}}

\definecolor{MyGrey}{rgb}{.804,.804,.756} 
\usepackage[colorlinks=false,linkbordercolor=MyGrey,citebordercolor=MyGrey,
urlbordercolor=MyGrey, pdfauthor={V. Lebed}, pdftitle={Cohomology of finite monogenic self-distributive structures}]{hyperref}
\usepackage{caption}
\usepackage[all]{hypcap}

  \theoremstyle{plain}
\newtheorem{theorem}{Theorem}
\newtheorem*{theorem*}{Theorem A}

\newtheorem{proposition}{Proposition}[section]
\newtheorem{corollary}[proposition]{Corollary}
\newtheorem{lemma}[proposition]{Lemma}
  \theoremstyle{remark}
\newtheorem{remark}[proposition]{Remark}
  \theoremstyle{definition}
\newtheorem{definition}[proposition]{Definition}

\newtheorem{example}[proposition]{Example}

\begin{document}
  \title{Cohomology of finite monogenic\\ self-distributive structures}
  \author{Victoria LEBED}{University of Nantes}
  \email{lebed.victoria@gmail.com}
  \subjclass[2010]{20N02, 55N35, 57M27}
  \keywords{self-distributivity, Laver tables, cyclic racks, rack cohomology}

  \maketitle

  \begin{abstract}
A shelf is a set with a binary operation~$\op$ satisfying $a \op (b \op c) = (a \op b) \op (a \op c)$. Racks are shelves with invertible translations $b \mapsto a \op b$; many of their aspects, including cohomological, are better understood than those of general shelves. Finite monogenic shelves (FMS), of which Laver tables and cyclic racks are the most famous examples, form a remarkably rich family of structures and play an important role in set theory. We compute the cohomology of FMS with arbitrary coefficients. On the way we develop general tools for studying the cohomology of shelves. Moreover, inside any finite shelf we identify a sub-rack which inherits its major characteristics, including the cohomology. For FMS, these sub-racks are all cyclic.
  \end{abstract}

  \section{Introduction}

A \bi{self-distributive structure}, or \bi{shelf}\footnote{The former term is used in set theory, while the latter, coined by Alissa Crans, belongs to the topologists' vocabulary.}, is a set~$S$ endowed with a binary operation~$\op$ satisfying the \emph{(left) self-distributivity} relation\footnote{Following the set-theoretical and algebraic traditions, we use the left version of self-distributivity rather than the right one, more common in knot theory. The two are obviously equivalent.}
\begin{equation}\label{E:SD}
a \op (b \op c) = (a \op b) \op (a \op c).
\end{equation}
A shelf is called a \bi{rack} if for any $a \in S$, the map $x \mapsto a \op x$ is a bijection from~$S$ to itself; an idempotent rack (in the sense that $a \op a = a$ for all~$a$) is called a \bi{quandle}. Groups with the conjugation operation $a \op b = aba^{-1}$ are major examples of quandles. Self-distributive structures sporadically emerged in mathematics starting from the late 19th century. However, their systematic investigation had to wait until the 1980s, when spectacular applications to knot classification, large cardinal study, universal algebra questions, and, later, Hopf algebra classification independently brought them into the spotlight of different mathematical communities. 

This article focuses on \bi{finite monogenic shelves} (\bi{FMS}), i.e., finite shelves generated by a single element with respect to~$\op$. Their properties are very different from those of conjugation quandles, and they form an extremely rich class, often described as ``combinatorially chaotic''. However, Ale\v{s} Dr\'{a}pal \cite{DraAlg,DraGro} (see also \cite{Dehornoy2,Smedberg}) found a way to construct them all out of the following two basic families: 

\begin{enumerate}
\item The \bi{Laver table}~$A_n$ (here $n \ge 0$) is the unique shelf $(\{1,2,3,\ldots,2^n\},\,\op)$ satisfying the initialization relation
\begin{equation}\label{E:Init}
a \op 1 \equiv a+1 \mod 2^n.
\end{equation}
Richard Laver \cite{Laver} discovered these structures as a by-product of the study of ite\-rations of elementary embeddings in set theory, showing in particular that properties~\eqref{E:SD}-\eqref{E:Init} uniquely determine~$\op$. An extensive study of the combinatorics of Laver tables followed, unveiling their intricate properties, some of which are currently established only under a strong set-theoretical axiom. 
The~$A_n$ are monogenic ($1$ is the unique generator), and are not racks unless $n = 0$.  

\item The \bi{cyclic shelf}~$C_{r,m}$ (here $r \ge 0$, $m \ge 1$) is the shelf $(\{-r,-(r-1),\ldots,$ $m-2,m-1\},\,\op)$ with\footnote{Here we follow the non-standard but convenient notations of Matthew Smedberg \cite{Smedberg}.} 
\begin{equation}\label{E:CShelf}
a \op b = \begin{cases} b+1 &\text{ if } b \neq m-1,\\ 0  &\text{ if } b = m-1. \end{cases}
\end{equation}
The shelf~$C_{r,m}$ is generated by its element~$-r$; it is a rack if and only if $r=0$, in which case it is called a \emph{cyclic rack}\footnote{This term can also refer to the infinite rack $(\ZZ, a \op b = b+1)$, which is not considered here.}.
\end{enumerate}

Finite shelves are presented by their multiplication tables, containing the value of~$p \op q$ in the cell~$(p,q)$ (see Table~\ref{F:First} for examples).

\begin{center}
\begin{tabular}{c|cccccccc}
$A_3$&$1$&$2$&$3$&$4$&$5$&$6$&$7$&$8$\\
\hline
$1$&$2$&$4$&$6$&$8$&$2$&$4$&$6$&$8$\\
$2$&$3$&$4$&$7$&$8$&$3$&$4$&$7$&$8$\\
$3$&$4$&$8$&$4$&$8$&$4$&$8$&$4$&$8$\\
$4$&$5$&$6$&$7$&$8$&$5$&$6$&$7$&$8$\\
$5$&$6$&$8$&$6$&$8$&$6$&$8$&$6$&$8$\\
$6$&$7$&$8$&$7$&$8$&$7$&$8$&$7$&$8$\\
$7$&$8$&$8$&$8$&$8$&$8$&$8$&$8$&$8$\\
$8$&$1$&$2$&$3$&$4$&$5$&$6$&$7$&$8$\\
\end{tabular}
\qquad 
\begin{tabular}{c|cccccccc}
$C_{3,5}$&$-3$&$-2$&$-1$&$0$&$1$&$2$&$3$&$4$\\
\hline
$-3$&$-2$&$-1$&$0$&$1$&$2$&$3$&$4$&$0$\\
$-2$&$-2$&$-1$&$0$&$1$&$2$&$3$&$4$&$0$\\
$-1$&$-2$&$-1$&$0$&$1$&$2$&$3$&$4$&$0$\\
$0$&$-2$&$-1$&$0$&$1$&$2$&$3$&$4$&$0$\\
$1$&$-2$&$-1$&$0$&$1$&$2$&$3$&$4$&$0$\\
$2$&$-2$&$-1$&$0$&$1$&$2$&$3$&$4$&$0$\\
$3$&$-2$&$-1$&$0$&$1$&$2$&$3$&$4$&$0$\\
$4$&$-2$&$-1$&$0$&$1$&$2$&$3$&$4$&$0$\\
\end{tabular}
\captionof{table}{Multiplication tables for the Laver table~$A_3$ and the cyclic shelf $C_{3,5}$}\label{F:First}
\end{center}

Important advances in knot-theoretic and Hopf-algebraic applications were achieved using the \bi{cohomological approach to self-distributivity}, initiated in \cite{RackHom,QuandleHom} and further developed in~\cite{AndrGr}. Here is how it works. For a shelf $(S,\op)$ and an abelian group~$A$, denote by $C^k(S,A)$ the abelian group of maps from~$S^{\times k}$ to~$A$ (where the group operation is the point-wise summation), and put
\begin{align}
(d^k \phi)(a_1, \ldots, a_{k+1}) = \sum_{i=1}^{k} (-1)^{i-1} (& \phi(a_1,\ldots,a_{i-1},a_i \op a_{i+1}, \ldots, a_i \op a_{k+1})\qquad\label{E:RackCohom}\\
& - \phi(a_1,\ldots,a_{i-1},a_{i+1}, \ldots,a_{k+1})).\notag
\end{align}
The definition is completed by $C^0(S,A) = A$ and $d^0 = 0$. One checks that $(C^k(S,A),d^k)$ is a cochain complex. Its cocycle / coboundary / cohomology groups are denoted by $Z^k(S,A)$, $B^k(S,A)$, and $H^k(S,A)$ respectively.\footnote{The described cohomology theory is called \emph{rack cohomology}, which is reflected in the classical notations $Z^k_\Rack$, $B^k_\Rack$, $H^k_\Rack$. Here we consider only this theory, and hence simplify notations and names.} Coloring techniques from~\cite{QuandleHom} produce an invariant of positive braids out of a shelf $(S, \op)$ equipped with a $2$- or $3$-cocycle~$\phi$. That invariant extends to arbitrary braids if~$S$ is a rack, and to knots and links if~$S$ is a quandle and~$\phi$ satisfies one additional condition. More generally, $(k+1)$-cocycles are used in the study of $k$- or $(k-1)$-dimensional braids and knots. 

Topological and algebraic applications explain why the \emph{cohomology of racks} receives so much attention. One of the major results belongs to Pavel Etingof and Mat\'{i}as Gra{\~n}a \cite{EG_Betti}. They showed that for a rack $(S, \op)$, $\dim_\QQ (H^k(S,\QQ)) = |\Orb(S)|^k$, where $\Orb(S)$ is the set of \emph{orbits}, i.e., classes for the equivalence relation on~$S$ induced by $a \sim b \op a$. The \emph{indicator functions of $k$-tuples of orbits} can be chosen as generators; in particular, the class of any non-zero constant map is a free generator in the monogenic case. This is bad news for knot theorists, since cohomologous cocycles yield the same invariant, and the invariants obtained from orbit indicator functions contain nothing more than linking numbers. However, in general the group $H^k(S,\ZZ)$ may involve torsion even for the most basic quandles, producing interesting invariants -- see for example \cite{Carter_Betti,M_Betti,LN_Betti,PrzHomDihedral}.

\bi{The cohomological aspects of non-rack shelves} have, on the contrary, remained in the shadow until recently, probably because current methods extract only positive braid invariants out of them. However, the example of free shelves (which are conjecturally approximated by Laver tables) confirms that general shelf colorings  may be adapted to arbitrary braids, yielding extremely strong invariants \cite{DehOrder2,Dehornoy2}. This led Patrick Dehornoy to launch a challenging project of developing braid-theoretic applications of Laver tables \cite{DehLaverSurvey}. As a first step, Dehornoy and the author \cite{DehLeb} explicitly described $Z^k(A_n,\ZZ)$, $B^k(A_n,\ZZ)$, and $H^k(A_n,\ZZ)$ for $k \le 3$, revealing in particular rich combinatorics behind the $2$-cocycles of the~$A_n$. 

In Section~\ref{S:Proj} of the present paper, we extend the cohomology calculations of~\cite{EG_Betti} from racks to a wider class of shelves, comprising Laver tables, cyclic shelves, and more sophisticated classes of FMS mixing these two. In particular, we obtain 
$$\dim_\QQ (H^k(A_n,\QQ))=\dim_\QQ (H^k(C_{r,m},\QQ))=1$$ 
for all~$k$.  We push the techniques from~\cite{EG_Betti} further to get a better structural understanding of the complex $(C^k(S,A),d^k)$, presenting it as
\begin{align*}
(C^k(S,A),d^k) &\;\cong \; (C^k(\Orb(S),A),0) \;\bigoplus \; (\text{an acyclic complex})
\end{align*}
for certain abelian groups~$A$ (in particular, $A=\QQ$). In the case of Laver tables, this method is especially powerful: it works for any abelian group $A$.

In Section~\ref{S:Retract}, we introduce the notion of \bi{retracts} of a shelf $(S,\op)$. These are certain sub-shelves of~$S$ sharing its major characteristics (cohomology, the  number of orbits, etc.). We show that a finite shelf admits retracts, and all its minimal retracts are in fact pairwise isomorphic racks. The  rack retracts of all FMS turn out to be cyclic. To understand the cohomology of  FMS, it thus suffices to compute all the groups $H^k(C_{r,m},A)$, which is done very explicitly in Section~\ref{S:Cyclic}.  Section~\ref{S:Laver} contains similar explicit calculations for Laver tables; together with the $H^k(A_n,A)$, already determined in Section~\ref{S:Proj}, they yield the cocycle and coboundary groups of the~$A_n$. All these computations are summarized as follows:
\begin{theorem*}
Suppose that $(S,\op)$ is a finite monogenic shelf, and~$A$ is an abelian group. Take a $k \ge 0$. Then one has
\begin{enumerate}
\item $H^k(S,A) \cong A$.
\end{enumerate}
If~$S$ is either a Laver table or a cyclic shelf, one has moreover
\begin{enumerate}[resume]
\item $B^k(S,A) \cong A^{P_k(|S|)}$, where~$P_k$ is the polynomial 
\begin{equation}\label{E:Pk}
\displaystyle P_k(x)=\frac{x^k-x^{\,k \hspace{-0.5ex} \mod \hspace{-0.2ex} 2}}{x+1};
\end{equation}
\item $Z^k(S,A) \cong H^k(S,A) \oplus B^k(S,A)$.
\end{enumerate}
\end{theorem*}
These isomorphisms are made explicit and are proved in Sections \ref{S:Cyclic}-\ref{S:Laver}.

Note that for cyclic shelves, our theorem is stronger than what one might expect to get using Etingof-Gra{\~n}a's approach, since, as we show, the classes of constant maps no longer generate $H^k(C_{r,m},\ZZ)$ in general. 

\textbf{Acknowledgements.} The author is grateful to Patrick Dehornoy for an introduction to the fascinating world of Laver tables, and to Seiichi Kamada  for encouraging comments on this work. The hospitality of OCAMI (Osaka City University) and Henri Lebesgue Center (University of Nantes), where the paper was written, deserve a special mention. The author also greatly appreciates the support of a JSPS Postdoctoral Fellowship For Foreign Researchers, JSPS KAKENHI Grant 25$\cdot$03315, and program ANR-11-LABX-0020-01.

\section{Cohomology of shelves with a strong projector}\label{S:Proj}  

The aim of this section is to adapt the rack cohomology computations from~\cite{EG_Betti} to certain non-rack shelves, and to sharpen them on the way.

Take a shelf $(S, \op)$ and a module~$A$ over a commutative ring~$R$ (for instance, an abelian group~$A$ viewed as a $\ZZ$-module). 

The set $C^k(S,A)$ of maps from~$S^{\times k}$ to~$A$ is an $R$-module, with the operations defined point-wise. Thus $R$-linear maps from $C^k(S,A)$ to itself (which we write on the right of their arguments) form an $R$-algebra 
$$M^k = \End_R(C^k(S,A)).$$ 
For any $a \in S$, consider the \emph{left translation} map~$\tau_a$ and its induced action on $C^k(S,A)$:
\begin{align*}
\tau_a &\in \End(S),  & T^k_a &\in M^k,\\
b &\mapsto a \op b;   & \phi &\mapsto (\phi \circ \tau_a^{\times k} \colon (b_1,\ldots,b_k) \mapsto \phi(a \op b_1,\ldots, a \op b_k)).
\end{align*}
Here $\End(S)$ is the set of \textit{shelf morphisms} (i.e., maps preserving the shelf operation) from~$S$ to itself. Remark that for non-rack shelves, the maps~$\tau_a$ and~$T^k_a$ are not ne\-cessarily invertible.
Further, let~$T=T_S$ be the sub-semigroup of $\End(S)$ generated by all the~$\tau_a$, and let~$R T$ be its (non-unital) semigroup algebra. Consider the \textit{diagonal} and \textit{augmentation} $R$-algebra maps
\begin{align*}
\pi^k \colon R T &\to M^k,  & \varepsilon \colon R T &\to R,\\
\tau_a &\mapsto T^k_a;  & \tau_a &\mapsto 1.
\end{align*}
We write $t \op b$ for a $t \in T$ applied to a $b \in S$. When bi-linearized, this operation gives a linear action of~$RT$ on~$RS$, for which we keep the notation~$\op$.
We also write $\phi \cdot p$ for $\pi^k(p)$, $p \in RT$, applied to a $\phi \in C^k(S,A)$.
 
The \emph{$S$-invariant part} of $C^k(S,A)$ is defined by
\begin{align*}
C^k_{inv}(S,A) & = \big\{\, \phi\colon S^{\times k} \rightarrow A \,\big|\, \phi \cdot \tau_a = \phi \text{ for all } a \in S \,\big\}.
\end{align*}
For all $\phi \in C^k_{inv}(S,A)$ and $t \in R T$, one has $\phi \cdot t = \varepsilon(t) \phi$.

In what follows, bold letters $\oa$, $\ob$, etc. will stand for $k$-tuples from $S^{\times k}$, with~$k$ determined by the context. Given a $b \in S$ and a $\phi \in C^k(S,A)$, define $\phi_b \in C^{k-1}(S,A)$ as the \emph{partial evaluation} $\phi_b\colon \ob \mapsto \phi(b,\ob)$. This yields a map $ev_b \colon C^k(S,A) \to C^{k-1}(S,A), \, \phi \mapsto \phi_b$. The following easy observation describes the interactions between partial evaluations, differentials, and left translations:

\begin{lemma}\label{L:CohomPropertiesPre}
For any map $\phi \in C^k(S,A)$ and any $b \in S$, one has
\begin{enumerate}
\item $d^k (\phi \cdot \tau_a) = (d^k \phi) \cdot \tau_a$;
\item $ \phi \cdot \tau_b - \phi = (d^k \phi)_b + d^{k-1} (\phi_b)$;
\item $ (\phi \cdot \tau_a)_b = \phi_{a \op b} \cdot \tau_a$.
\end{enumerate}
\end{lemma}

From this, one deduces
\begin{lemma}\label{L:CohomProperties}
\begin{enumerate}
\item The algebra~$R T_S$ acts on~$C^\bullet(S,A)$ via~$\pi^\bullet$ by morphisms of complexes.
\item The induced $R T_S$-action on the cohomology~$H^\bullet(S,A)$ is trivial.
\item The $S$-invariant part $C^\bullet_{inv}(S,A)$ forms a sub-complex of $C^{\bullet}(S,A)$.
\item For any $b \in S$, one has $ev_b(Z^\bullet_{inv}(S,A)) \subseteq Z^{\bullet}(S,A)$.
\end{enumerate}
\end{lemma}

To terminate this long list of preliminaries, we recall the equivalence relation~$\sim$ on~$S$ (sometimes denoted by~${\sim_S}$ to avoid ambiguities) induced by $a \sim b \op a$. It divides~$S$ into classes, called \emph{orbits}. Their set, denoted by~$\Orb(S)$, receives an induced shelf structure which is  trivial: $\OO' \op \OO = \OO$. 

The following definition is central to this section.

\begin{definition}\label{D:proj}
A \emph{semi-strong projector} for a shelf $(S,\op)$ over a ring~$R$ is a $P \in R T_S$ which is
\begin{enumerate}
\item normalized: $\varepsilon (P)= 1$;
\item right $S$-invariant, i.e., $P \tau_a = P$ for all $a \in S$.
\end{enumerate}
It is called a \emph{strong projector} if it is moreover 
\begin{enumerate}[resume]
\item left $S$-invariant, i.e., $\tau_a P = P$ for all $a \in S$.
\end{enumerate}
\end{definition}   

The existence of a (semi-)strong projector for a shelf~$S$ heavily depends on the ring~$R$ one works with, as we will see in examples.

\begin{lemma}\label{L:proj}
\begin{enumerate}
\item A semi-strong projector is indeed a projector.
\item If~$S$ admits a strong projector, it is unique even among semi-strong projectors. 
\item A semi-strong projector acts on elements $b,b'$ from the same orbit of~$S$ in the same way, in the sense of $P \op b = P \op b' \in RS$.
\end{enumerate}
\end{lemma}

\begin{proof}
Property~$2$ (or~$3$) defining projectors implies that for all $t \in R T$, one has $P t =\varepsilon(t) P$ (respectively, $t P =\varepsilon(t) P$). In particular,
\begin{enumerate}
\item For a semi-strong projector $P$, one has $ PP =  \varepsilon(P) P=P$. 
\item For a strong projector $P$ and a semi-strong projector $P'$, one has $ P' = \varepsilon(P) P'=P'P = \varepsilon(P') P = P$.
\end{enumerate}
The last point follows from $P \op b = (P \tau_a) \op b = P \op (a \op b)$ for all $a,b \in S$.
\end{proof}

The presence of a (semi-)strong projector considerably simplifies the study of the cohomology of our shelf, since it makes possible an adaptation of the key results of~\cite{EG_Betti}, together with their proofs.

\begin{proposition}\label{P:CohomDecompose}
Let $(S,\op)$ be a shelf admitting a semi-strong projector~$P$ over~$R$, and let~$A$ be an $R$-module.
\begin{enumerate}
\item The complex $C^\bullet(S,A)$ is then a direct sum of sub-complexes
\begin{align}\label{E:Decompose1}
C^\bullet(S,A) &= C^\bullet(S,A) \cdot P \oplus C^\bullet(S,A)\cdot(1-P).
\end{align}
\item The sub-complex $C^\bullet(S,A) \cdot P$ coincides with $C^\bullet_{inv}(S,A)$.
\item The sub-complex $C^\bullet(S,A) \cdot (1-P)$ is acyclic.
\item The complex inclusion $C^\bullet_{inv}(S,A) \hookrightarrow C^\bullet(S,A)$ induces an isomorphism in cohomology.
\end{enumerate}
\end{proposition}

\begin{proof}
Point~$1$ follows from Lemmas~\ref{L:proj} and~\ref{L:CohomProperties}, Point~$2$ is a consequence of definitions, and Point~$4$ summarizes the preceding ones. Let us prove Point~$3$. Take a cocycle $c \cdot (1-P)$ in  $C^k(S,A) \cdot (1-P)$. According to Lemma~\ref{L:CohomProperties}, modulo coboundaries one has
$(c \cdot (1-P))\cdot P =  \varepsilon(P) c \cdot (1-P) $, which is simply $c \cdot(1-P)$. Compare this with
$(c \cdot (1-P))\cdot P = c \cdot (1-P) P = c \cdot (P-P) = 0$ to deduce the triviality of~$c \cdot (1-P)$ in cohomology.
\end{proof}

The proposition can fail for shelves without semi-strong projectors; a counter-example will be given in Section~\ref{S:Cyclic}.

The cohomology of a shelf with a \textit{strong} projector can be described very explicitly: 

\begin{theorem}\label{T:CohomComput}
Let $(S,\op)$ be a shelf admitting a strong projector over~$R$, and let~$A$ be an $R$-module. Then for all $k \ge 0$, one has the following morphism of $R$-modules:
\begin{align}
H^k(S,A) \cong &\, C^k(\Orb(S),A) = A^{\Orb(S)^{\times k}},\notag\\
[\phi \circ pr^{\times k}] \mapsfrom & \; \phi,\label{E:HkWithProj}
\end{align}
where the projection $pr \colon S \twoheadrightarrow \Orb(S)$ sends an element of~$S$ to its orbit.
\end{theorem}

\begin{remark}
If~$A$ is moreover an $R$-algebra, then~$H^k(S,A)$ (and thus $H^k_{inv}(S,A)$) is a free $A$-module, with a basis given by the classes of the \emph{orbit indicator functions}, indexed by $(\OO_1, \ldots, \OO_k) \in \Orb(S)^{\times k}$:
\begin{align*}
\delta_{\OO_1, \ldots, \OO_k}(\oa) &= \begin{cases} 1 & \text{ if } a_i \in \OO_i, 1 \le i \le k,\\
0 & \text{ otherwise}. \end{cases}
\end{align*}
\end{remark}

\begin{remark}\label{R:mono}
If~$S$ is monogenic, then it has a single orbit, and the theorem describes $H^k(S,A)$ as containing the classes of constant maps only (with different constants giving different classes). If~$A$ is moreover an $R$-algebra, then the $A$-module $H^k(S,A) \cong A$ is freely generated by the class of the constant map $\phi^k_{const} \colon \oa \mapsto 1$.
\end{remark}

\begin{proof}
Suppose $k \ge 2$, the case $k \le 1$ being easy. Together with $\pi^k\colon R T \rightarrow M^k$, consider its ``partial versions'' $\pi^{i,k-i}\colon R T \rightarrow M^k$, $0 \le i \le k$. These are $R$-algebra maps sending each $\tau_a$, $a \in S$, to $T^{i,k-i}_a \colon \phi \mapsto \phi \circ (\Id_S^{\times i} \times \tau_a^{\times (k-i)})$. Put $P^{i,k-i}=\pi^{i,k-i}(P)$, and $P^{k}=P^{0,k}=\pi^k(P)$. They are projectors since~$P$ is so. They behave nicely with respect to partial evaluations, e.g., one has a useful property
\begin{align}\label{E:Partials}
((\phi)P^{i,k-i})_b &= (\phi_b)P^{i-1,k-i}
\end{align}
for any $i \ge 1$, $b \in S$, and $\phi \in C^k(S,A)$. This property often reduces the study of $P^{i,k-i}$ to that of~$P^k$. For instance, we apply it to show that $P^{1,k-1}$ commutes with~$\pi^{k}(R T)$:
\begin{align}\label{E:PCommute}
&((\phi)P^{1,k-1} \cdot \tau_a)_b = ((\phi)P^{1,k-1})_{a \op b} \cdot \tau_a
= (\phi _{a \op b} \cdot P) \cdot \tau_a = \phi _{a \op b} \cdot P \\
&= (\phi _{a \op b} \cdot \tau_a ) \cdot P = (\phi \cdot \tau_a)_b \cdot P = ((\phi \cdot \tau_a)P^{1,k-1})_b \notag
\end{align}
(we used Point~$3$ of Lemma~\ref{L:CohomPropertiesPre} and the definition of a strong projector). 
In particular, $P^{1,k-1}$ commutes with $P^k$, yielding an endomorphism $P^{1,\bullet-1}$ of the complex $C^\bullet_{inv}(S,A) = C^\bullet(S,A) \cdot P$; indeed, $P^{1,\bullet-1}$ commutes with the differential since for any $b \in S$ and $\phi \in C^k_{inv}(S,A)$, one has
\begin{align*}
(d^k((\phi) P^{1,k-1}))_b &\overset{(\ast)}{=} - d^{k-1}(((\phi)P^{1,k-1})_b) =- d^{k-1}(\phi_b \cdot P)\\
& =- (d^{k-1}\phi_b ) \cdot P \overset{(\ast\ast)}{=} (d^k \phi)_b \cdot P = ((d^k \phi)P^{k,1})_b
\end{align*}
(in $(\ast)$ and $(\ast\ast)$ we used Point~$2$ of Lemma~\ref{L:CohomPropertiesPre} and the $S$-invariance of~$\phi$ and~$(\phi) P^{1,k-1}$). Our $S$-invariant complex thus decomposes as
\begin{align}\label{E:Decompose2}
C^\bullet_{inv}(S,A) & = (C^\bullet_{inv}(S,A))P^{1,\bullet-1} \oplus (C^\bullet_{inv}(S,A))(1-P^{1,\bullet-1}).
\end{align}

Let us first study the sub-complex $(C^\bullet_{inv})P^{1,\bullet-1} = (C^\bullet \cdot P)P^{1,\bullet-1} = (C^\bullet) P^{1,\bullet-1} \cdot P$. A map~$\phi$ from that complex and all its partial evaluations~$\phi_b$ are $S$-invariant (cf.~\eqref{E:Partials}). We will show that the value~$\phi(b,\ob)$ depends on the orbit of the first coordinate~$b$, and not on~$b$ itself. Indeed, for any $a \in S$, one has 
\begin{align}\label{E:ConstOnOrbit}
\phi_{a \op b}  = \phi_{a \op b} \cdot \tau_a = (\phi \cdot \tau_a)_b = \phi_b.
\end{align}
 This yields an isomorphism of complexes
\begin{align}
(C^\bullet_{inv}(S,A))P^{1,\bullet-1} & \overset{\sim}{\longrightarrow} \bigoplus_{\OO \in \Orb(S)} C^{\bullet-1}_{inv}(S,A),\label{E:DecomposeOrbits}\\
\phi &\longmapsto ((-1)^{\bullet}\phi_{b_{\OO}})_{\OO \in \Orb(S)},\notag
\end{align}
where~$b_{\OO}$ is an arbitrary representative of the orbit~$\OO$; this is indeed a map of complexes, since Lemma~\ref{L:CohomPropertiesPre} ensures that $(d^k \phi)_{b_{\OO}}  = - d^{k-1} \phi_{b_{\OO}}$.

We next show that the complex $(C^\bullet_{inv})(1-P^{1,\bullet-1})$ is acyclic; the theorem then follows by induction, using Proposition~\ref{P:CohomDecompose}. Take a cocycle~$\phi$ in $(C^k_{inv})(1-P^{1,k-1})$; it is $S$-invariant, and all~$\phi_b$ lie in $C^{k-1} \cdot (1-P)$ (cf.~\eqref{E:Partials}). Lemma~\ref{L:CohomProperties} gives $d^{k-1}\phi_b = -(d^k \phi)_b = 0$, thus $\phi_b$ is a coboundary (Proposition~\ref{P:CohomDecompose}), say, $\phi_b = d^{k-2} \psi^{(b)}$. Assemble these $\psi^{(b)}$ into a $\psi \in C^{k-1}$ defined by $\psi_b = \psi^{(b)}$. We will now present~$\phi$ as~$-d^{k-1}((\psi)P)$, implying its triviality in cohomology. Take a $t \in T$. One calculates
\begin{align*}
d^{k-2}((\psi \cdot t)_b) &= d^{k-2} (\psi_{t \op b} \cdot t) = (d^{k-2} \psi_{t \op b}) \cdot t = \phi_{t \op b} \cdot t = (\phi \cdot t)_b = \phi_b,
\end{align*}
implying $d^{k-2}((\psi\cdot P)_b) = \varepsilon(P)\phi_b = \phi_b$. Since $\psi \cdot P \in C^{k-1} \cdot P = C^{k-1}_{inv}$ is $S$-invariant, Lemma~\ref{L:CohomPropertiesPre} gives $d^{k-2}((\psi \cdot P)_b) =-(d^{k-1}(\psi \cdot P))_b$, hence $\phi = -d^{k-1}(\psi \cdot P)$ as desired.
\end{proof}

\begin{remark}\label{R:SmallGroups}
For $k \le 2$, the map~\eqref{E:HkWithProj} provides a description of $H^k_{inv}(S,A)$ even for shelves without semi-strong projectors.
\end{remark}

\begin{remark}\label{R:ZB}
An inductive argument using the isomorphism of complexes~\eqref{E:DecomposeOrbits} allows one to continue decompositions~\eqref{E:Decompose1} and~\eqref{E:Decompose2}, leading to the complex decomposition
\begin{align}\label{E:Decompose3}
C^\bullet(S,A) & \;=\; (C^\bullet(S,A))\wP^{\bullet}\; \bigoplus \; (\text{an acyclic sub-complex}),
\end{align}
where $\wP^k = P^k P^{1,k-1} P^{2,k-2} \cdots P^{k-1,1} \in M^k$. An argument repeating~\eqref{E:PCommute} shows that any~$P^{i,k-i}$ and~$P^{j,k-j}$ commute; mimicking~\eqref{E:ConstOnOrbit}, one deduces that a~$\phi$ lies in $(C^k(S,A))\wP^{k}$ if and only if $\phi(b_1,\ldots,b_k)=\phi(b'_1,\ldots,b'_k)$ whenever for all~$i$, $b_i$ and~$b'_i$ belong to the same orbit. From the definition~\eqref{E:RackCohom}, one immediately sees that $d^k \phi =0$ for such a~$\phi$. This transforms~\eqref{E:Decompose3} into
\begin{align}\label{E:Decompose4}
(C^\bullet(S,A),d^\bullet) & \;\cong\; (C^\bullet(\Orb(S),A),0) \; \bigoplus \; (\text{an acyclic complex}),
\end{align}
rendering the description~\eqref{E:HkWithProj} of the cohomology of~$S$ more precise. In particular, one sees that the cocycle $R$-modules split: 
\begin{align}
Z^k(S,A) &\cong H^k(S,A) \oplus B^k(S,A).\label{E:Decompose5}
\end{align}
\end{remark}

Let now~$S$ be finite. The theorem gives $\dim_{\QQ}(H^k(S,\QQ)) = |\Orb(S)|^k$. Together with $\dim_{\QQ}(C^k(S,\QQ)) = |S|^k$, one obtains (for instance by induction)
\begin{corollary}\label{C:ZB}
For a finite shelf $(S,\op)$ admitting a strong projector, one has  
\begin{align*}
\dim_{\QQ}(Z^k(S,\QQ)) &= P_k(|S|) + P_{k+1}(|\Orb(S)|)+1,\\
\dim_{\QQ}(B^k(S,\QQ)) &= P_k(|S|) - P_{k}(|\Orb(S)|)
\end{align*} 
for all $k \ge 0$. (See~\eqref{E:Pk} for the definition of the polynomials~$P_k$.)
\end{corollary}
In the finite monogenic case, these formulas read
$$\dim_{\QQ}(Z^k(S,\QQ)) = P_k(|S|)+1, \qquad \dim_{\QQ}(B^k(S,\QQ)) = P_k(|S|).$$

We next turn to examples, which include racks and basic finite monogenic shelves (cf. Introduction).

\begin{example}\label{Ex:Rack} 
For a rack~$S$, all translations $\tau_a$, $a \in S$ are invertible. Using this, $S$ can be shown to admit a semi-strong projector~$P$ if and only if its translation semigroup~$T_S$ is finite (which holds for instance for finite racks). In this case $T_S$ is a group, and $P=\frac{1}{|T_S|} \sum_{t \in T_S}t$, which is a strong projector. This~$P$ works  for any coefficient ring~$R$ where~$|T_S|$ is invertible. One recovers the cohomology calculations from~\cite{EG_Betti}.
\end{example}

\begin{example}\label{Ex:Cyclic} 
Take the cyclic shelf $C_{r,m}$, $r \ge 0$, $m \ge 1$. Its translation semigroup~$T_{C_{r,m}}$ is generated by the map $\theta=\tau_0 \colon b \mapsto b+1$, with the convention $(m-1)+1 = 0$. More precisely, 
\begin{align*}
T_{C_{r,m}} &\cong {\raisebox{1.5mm}{$<\theta>$} \big/ \raisebox{-1.5mm}{$(\theta^{r+m} = \theta^{r})$}} \, .
\end{align*} 
The only possibility for a strong projector over~$R$ is $P = \frac{1}{m} \sum_{i=0}^{m-1}\theta^{r+i}$, which works if the ring~$R$ contains $\frac{1}{m}$. This includes the case $R=\QQ$. For a module~$A$ over such an~$R$ and for all~$k \ge 0$, Theorem~\ref{T:CohomComput} yields $H^k(C_{r,m},A) \cong A$. When~$A$ is an $R$-algebra, this becomes an $A$-module freely generated by $[\phi^k_{const}]$ (Remark~\ref{R:mono}). However, the classes of constant maps do \textit{not} generate all integral cohomology groups $H^k(C_{r,m},\ZZ)$, as we will see in Section~\ref{S:Cyclic}.
\end{example}

\begin{example}\label{Ex:Laver} 
Consider next the Laver table~$A_n$, $n \ge 0$. We will use the very special properties of its elements~$2^n$ and~$2^n-1$:
\begin{enumerate}
\item The element $2^n$ is central\footnote{The term \textit{central element} is inspired by conjugation quandles.} for~$\op$: for all~$a\in A_n$, one has
\begin{align}
2^n \op a &= a,\label{E:LT1} \\
a \op 2^n &= 2^n.\label{E:LT2}
\end{align}
\item The left translation $\tau_{2^n-1}$ is a projector to~$2^n$: for all~$a\in A_n$, one has
\begin{align}
(2^n-1) \op a &=2^n.\label{E:LT3}
\end{align}
\end{enumerate}
For a proof, see for instance~\cite{Dehornoy2}; the example from Figure~\ref{F:First} can serve as an illustration. The relations above imply that the translation $\tau_{2^n-1}$ satisfies the absorption property $\tau_a \tau_{2^n-1} = \tau_{2^n-1} \tau_a = \tau_{2^n-1}$ for all $a\in A_n$. Therefore, $P=\tau_{2^n-1}$ is a strong projector in~$R T_{A_n}$ for any commutative ring~$R$. Remark~\ref{R:mono} and property~\eqref{E:Decompose5} then yield Points~$1$ and~$3$ of Theorem~A for Laver tables. Note that in this case, $[\phi^k_{const}]$ generates the $k$th cohomology group for coefficients in any algebra~$A$ over any ring~$R$. 
\end{example}

\begin{example}\label{Ex:E} 
We proceed with a very general construction of finite shelves, comprising the two families above. Take an $n\ge 0$ and maps $\rho\colon A_n \rightarrow \NN \cup \{0\}$, $\mu\colon A_n \rightarrow \NN$  such that 
\begin{align}\label{E:DefE}
(\,b = c \op a \text{ in } A_n\,) \qquad &\Longrightarrow\qquad (\,\mu(b) \,|\, \mu(a) \quad\&\quad \rho(b) \le \rho(a)+1\,).
\end{align}
These data allow one to define the following shelf:
\begin{align}
&E^*_{n,\rho,\mu} = \big\{\, (a,i) \,\big|\, a \in A_n,\, 0 \le i < \rho(a)+\mu(a) \,\big\},\notag\\ &(a,i) \op (b,j) = (a \op b,\, j+1),\label{E:DefE2}
\end{align}
where~$a \op b$ utilizes the shelf operation of~$A_n$, and we set~$(c,\alpha+s\mu(c)) = (c,\alpha)$ for all~$c \in A_n$, $s \ge 1$, $\rho(c) \le \alpha < \rho(c)+\mu(c)$. The self-distributivity~\eqref{E:SD} follows from the defining condition~\eqref{E:DefE}; see~\cite{DraGro,Dehornoy2,Smedberg} for more details. Taking $n=0$, one recovers the cyclic shelf $C_{\rho(1),\mu(1)}$; choosing constant maps $\rho\colon a \mapsto 0$, $\mu\colon a \mapsto 1$, one gets the Laver table~$A_n$. Thus the $E^*_{n,\rho,\mu}$  ``interpolate'' between those two families. Every shelf of this general type has a single orbit, without necessarily being monogenic. Put $E_{n,\rho,\mu} = E^*_{n,\rho,\mu} \coprod \{(1,-1)\}$, and keep the definition~\eqref{E:DefE2}; one gets a shelf with a single generator~$(1,-1)$. Combining our results for cyclic shelves and Laver tables, one sees that $P=\frac{1}{\mu(2^n)} \sum_{i=1}^{\mu(2^n)} \tau_{(2^n-1,0)}^{\rho(2^n)+i}$ is a strong projector over a ring~$R$ containing $\frac{1}{\mu(2^n)}$ for both $E^*_{n,\rho,\mu}$ and $E_{n,\rho,\mu}$. For such an~$R$, one obtains $H^k(E^{(*)}_{n,\rho,\mu},A) \cong A$. When~$A$ is an $R$-algebra, this becomes an $A$-module freely generated by $[\phi^k_{const}]$.
\end{example}

We now study the existence of (semi-)strong projectors in more detail. In contrast with the examples above, there are shelves for which projectors do not exist or are not unique:

\begin{example}\label{Ex:Free}
Consider the shelf~$\FF$ freely generated by a single element~$\gamma$. Conditions $l(\gamma)=1$ and $l(a \op b) = l(b)+1$ for all $a,b \in \FF$ uniquely define a length function $l \colon \FF \to \NN$. Every $\tau_a$  increases~$l$ by~$1$. This induces an $\NN$-grading on the $R$-algebra $R T_{\FF}$, with $deg(\tau_a)=1$, which renders equality $P \tau_a = P$ impossible for non-zero~$P$. Hence the free monogenic shelf~$\FF$ admits no semi-strong projectors. 
\end{example}

\begin{example}\label{Ex:1Retract}
A set~$S$ endowed with the operation $a \op b = a$ is always a shelf. In this case the translation semigroup~$T_S$ consists of translations~$\tau_a \colon b \mapsto a$ only, each of which is a semi-strong projector. Their properly weighted linear combination (e.g., $\tau_a +\tau_{a'}-\tau_{a''}$) are semi-strong projectors as well. However, no strong projectors are available if~$S$ has at least two elements, because of the uniqueness property (Lemma~\ref{L:proj}). 
\end{example}

Most projectors we saw in examples were ``average-type''. This is not a mere coincidence, as we now explain.

\begin{definition}
A shelf $(S,\op)$ with a finite translation semigroup~$T_S$ is \emph{quasi-finite}.
\end{definition}

All finite shelves are clearly quasi-finite. The converse is not true:
\begin{example}
Take a finite set~$S$ and a map $f \colon S \to S$. Operation $a \op b = f(b)$ defines a shelf structure on~$S$. Take another (possibly infinite) set~$I$, fix an $s \in S$, and extend~$f$ to $S \sqcup I$ by $f(b) = s$, $b \in I$. Then $S \sqcup I$ with $a \op b = f(b)$ is a quasi-finite but not necessarily finite shelf.
\end{example}

\begin{definition}
A \emph{semi-projective family} for a shelf $(S,\op)$ is a finite sub-set~$T'$ of~$T_S$ on which the right multiplication by every~$\tau_a$ induces a permutation, i.e., $T' = T'\tau_a$ as sets. It is called a \emph{projective family} if moreover $T' = \tau_a T'$ for all~$\tau_a$.
\end{definition}

(Semi-)strong projectors can be described in terms of (semi-)projective families:

\begin{proposition}\label{P:ProjectorAverage}
Let $(S,\op)$ be a shelf.
\begin{enumerate}
\item One has the following trichotomy: 
\begin{description}
\item[Option A] $S$ has no semi-projective families;
\item[Option B] $S$ has at least two semi-projective families and no projective families;
\item[Option C] $S$ has only one semi-projective family~$T'$, which is in fact projective.
\end{description}
\item $S$ admits a strong projector~$P$ over a ring~$R$ if and only if Option~C holds and~$|T'|$ is invertible in~$R$. In this case, $P = \frac{1}{|T'|} \sum_{t \in T'}t$.
\item $S$ admits semi-strong non-strong projectors over a ring~$R$ if and only if Option~B holds and for some $k \in \NN$ there exist pairwise disjoint semi-projective families $T_1, \ldots, T_k$ of~$S$ and non-zero elements $\alpha_1, \ldots, \alpha_k$ of~$R$ satisfying $\sum_{i} |T_i| \alpha_i = 1$. In this case, $P = \sum_{i} \alpha_i \sum_{t \in T_i}t$ is a semi-strong projector. This yields a complete list of semi-strong projectors over~$R$ without repetition, if one requires the~$\alpha_i$ to be pairwise distinct.
\item A quasi-finite shelf $(S,\op)$ has a (possibly empty) collection of pairwise disjoint semi-projective families of the same size such that any semi-projective family of~$S$ is a union of some of these. The semigroup~$T_S$ acts transitively on this collection by the left multiplication. 
\end{enumerate}
\end{proposition}

\begin{definition}
The semi-projective families described in Point~4 are called \emph{atomic}. 
\end{definition}

\begin{proof}
To prove Point~1, use the following observations:
\begin{itemize}
\item a semi-projective family~$T'$ of~$S$ gives rise to the semi-strong projector $P = \frac{1}{|T'|} \sum_{t \in T'}t$ over~$\QQ$, which is a strong projector if~$T'$ is projective; thus a projective and a distinct semi-projective families cannot co-exist (Lemma~\ref{L:proj});
\item if a family~$T'$ is semi-projective, than so is $\tau_a T'$, so the uniqueness in Option~C forces $\tau_a T' = T'$ for all $a \in S$, hence the projectivity of~$T'$.
\end{itemize}

Now, take a $P \in R T_S$, and regroup its summands as $P = \sum_{i \in I} \alpha_i P_i$, where~$I$ is a finite set, the~$\alpha_i$ are pairwise distinct non-zero elements of~$R$, $P_i = \sum_{t \in T_i}t$, and the~$T_i$ are pairwise disjoint sub-sets of~$T_S$. Property $P \tau_a = P$ is equivalent to the right multiplication by~$\tau_a$ permuting the elements of each~$T_i$. Property $\tau_b P = P$ is analyzed in a similar way. This yields Points~2 and~3.

Next, remark that
\begin{itemize}
\item the union, the intersection, and the set difference of two semi-projective families is a semi-projective family;
\item given two semi-projective families $T',T''$ and a $t'' \in T''$, one gets a semi-projective family $t''T'$ contained in~$T''$.
\end{itemize}
Together with standard finiteness arguments, these observations imply Point~4.
\end{proof}

Most examples above realize Option~C, except for the free monogenic shelf~$\FF$ and the non-trivial shelves from Example~\ref{Ex:1Retract}, which illustrate Option~A and Option~B respectively. For the latter, the atomic families are the $\{\tau_a\}$ for $a \in S$.

In Proposition~\ref{P:RetractVsProj}, we will see that Option~A is impossible for quasi-finite shelves.

\section{Cohomology of shelves admitting a retract}\label{S:Retract}  

In this section, we show how to reduce rack cohomology calculation for some shelves to that for their ``nice'' sub-shelves. In particular, for any \textit{finite} shelf~$S$ we exhibit a sub-shelf which is in fact a \textit{rack}, and whose cohomology (as well as other major characteristics) is the same as that of~$S$. For Laver tables and cyclic shelves, the trivial one-element rack and, respectively, cyclic racks do the job. We also explore the uniqueness question for ``nice'' sub-shelves, and relate it to the existence of (semi-) strong %
 projectors.

\begin{definition}
A \emph{sub-shelf} of a shelf $(S,\op)$ is a $\op$-closed sub-set of~$S$. A sub-shelf~$S'$ of $(S,\op)$ is called its \emph{retract} if there exists an element $t$~in its translation semigroup~$T_S$, called \emph{retraction}, such that
\begin{enumerate}
\item the action of~$t$ projects~$S$ to~$S'$: $t \op S = S'$;
\item restricted to~$S'$, the action of~$t$ is trivial: $t \op b = b$ for all $b \in S'$.
\end{enumerate}
If moreover~$\op$ restricts to a rack operation on~$S'$, we talk about a \emph{rack retract}.
\end{definition}  

One easily verifies

\begin{lemma}\label{L:tt}
A $t \in T_S$ is a retraction for some retract if and only if $tt=t$ in $T_S$.
\end{lemma}

\begin{example}\label{Ex:RackRetract}
For a rack $(S,\op)$, all $t \in T_S$ are invertible in $\End(S)$, so~$S$ admits no retracts unless $\Id_S \in T_S$, in which case~$S$ is a retract of itself with $t=\Id_S$. 
\end{example}

\begin{example}
The free monogenic shelf~$\FF$ (Example~\ref{Ex:Free}) has no retracts at all: relation $tt=t$ is impossible in the $\NN$-graded semigroup~$T_{\FF}$. 
\end{example}

Things get more interesting for quasi-finite non-rack shelves:

\begin{theorem}\label{T:RetractRack}
A quasi-finite shelf $(S,\op)$ admits rack retracts. Any $u \in T_S$ restricted to a rack retract of $(S,\op)$ sends it isomorphically\footnote{A \emph{rack (iso-)morphism} between two racks is a shelf (iso-)morphism between the underlying shelf structures.} onto a rack retract of $(S,\op)$. This defines a transitive action of~$T_S$ on the set of all rack retracts of $(S,\op)$.
\end{theorem}

In particular, a finite shelf admits rack retracts which are all of the same size.

\begin{definition}
The rack retract of a quasi-finite shelf $(S,\op)$ (defined up to rack isomorphism) is called the \emph{rack type} of $(S,\op)$, and is denoted by $\MRR(S,\op)$.
\end{definition}

\begin{proof}
Start with an easy observation:

\begin{lemma}\label{L:RetractTrans}
Being a (rack) retract is a transitive relation, i.e., a (rack) retract of a (rack) retract of $(S,\op)$ is a (rack) retract of $(S,\op)$.
\end{lemma}

A quasi-finite rack is its own rack retract, with the retraction~$\Id_S$: indeed, $\Id_S$ can be presented as~$\tau_a^k$ for some~$k \in \NN$ and $a \in S$ because of the finiteness and the cancellation property of~$T_S$.

Now, suppose that $(S,\op)$ is not a rack, i.e., $\tau_a$ is not invertible for some $a \in S$. The quasi-finiteness of~$S$ implies $\tau_a^{k+s} = \tau_a^k$ for some $k,s \in \NN$, which yields $\tau_a^n\tau_a^n = \tau_a^n$ for a sufficiently large multiple~$n$ of~$s$. Thus $S' = \tau_a^n \op S$ is a retract of $(S,\op)$. It is strictly contained in~$S$: indeed, $S'=S$  would imply the surjectivity of $\tau_a \colon S \to S$, and 
$\tau_a^n\tau_a^n = \tau_a^n$ would then give its invertibility. Lemma~\ref{L:RetractTrans} allows an iteration of this argument. This yields a strictly decreasing sequence of retracts of~$S$, which has to be finite since all retracts have the form $t \op S$, with~$t$ belonging to the finite translation semigroup~$T_S$. Its last element is a rack retract of $(S,\op)$.

Note that we have just shown that for any $a \in S$, there is a retraction of the form $t\tau_a$ for some $t \in T_S$.

Next, take two rack retracts~$S'$ and~$S''$ of $(S,\op)$, with retractions~$t'$ and~$t''$ respectively. Present~$t''$ as $\tau_{a_1} \cdots \tau_{a_k}$  Put $t=t't''$, which can be written as
$$t=t't''= t'\tau_{a_1} \cdots \tau_{a_k} = \tau_{t' \op a_1} \cdots \tau_{t' \op a_k} t'.$$ 
This restricts to a bijection $S' \to S'$: indeed, $S' = t' \op S = t' \op S'$ is a rack, so all the $t' \op a_i \in S'$ yield bijective translations~$\tau_{t' \op a_i}$ on~$S'$. Thus~$t''$ restricts to a rack injection $S' \hookrightarrow S''$, and~$t'$ restricts to a rack surjection $S'' \twoheadrightarrow S'$. A symmetric argument shows that in fact~$t''$ yields a rack isomorphism $S' \overset{\sim}{\to} S''$.

Take now a rack retract~$S'$ of $(S,\op)$ with retraction~$t'$ and an $a \in S$. $S'' = a \op S'$ is a sub-rack of $(S,\op)$\footnote{A \emph{sub-rack} of a shelf $(S,\op)$ is its sub-shelf on which~$\op$ defines a rack structure.}. Above we showed the existence of a retraction of the form $t\tau_a$ and its injectivity on~$S'$; thus~$\tau_a$ restricts to a rack isomorphism $S' \overset{\sim}{\to} S''$. It remains to show that~$S''$ is a retract. Put $t'' = \tau_a t'$. The quasi-finiteness of $(S,\op)$ guarantees that $(t'')^n (t'')^n = (t'')^n$ for some $n \ge 1$. Writing~$t'$ as $\tau_{a_1} \cdots \tau_{a_k}$, one gets $(t'')^n = (\tau_{t'' \op a}\tau_{t'' \op a_1} \cdots \tau_{t'' \op a_k})^{n-1}t''$. Since $S'' = t'' \op S$ is a rack, $\tau_{t'' \op a}\tau_{t'' \op a_1} \cdots \tau_{t'' \op a_k}$ restricts to a bijection $S'' \to S''$. Consequently, $(t'')^n$ is a retraction to $(t'')^n \op S = S''$.
\end{proof}

The following theorem relates the cohomologies of a shelf and of its retract.

\begin{theorem}\label{T:RetractHom}
Let~$S'$ be a retract of a shelf $(S,\op)$. Take an $R$-module~$A$.
\begin{enumerate}
\item The cohomology $R$-modules $H^\bullet(S,A)$ and $H^\bullet(S',A)$ are isomorphic.
\item For {invariant} cohomology one can say more: $C^\bullet_{inv}(S,A)$ and $C^\bullet_{inv}(S',A)$ are isomorphic cochain complexes.
\end{enumerate}
In both situations, the isomorphisms are realized by the maps induced by the retraction $S \twoheadrightarrow S'$ and the shelf inclusion $S'\hookrightarrow S$.
\end{theorem}

\begin{proof}
A retraction~$t$ for~$S'$ induces a map $\tau \colon S \twoheadrightarrow S'$. The properties of~$t$ allow one to construct a commutative diagram of shelf morphisms
$$\xymatrix@!0 @R=1cm @C=2cm{
& S \ar@{->>}[dr]^{\tau} \ar[rr]^{t} & & S \\
S'\ar@{^{(}->}[ur]^i \ar[rr]^{\Id} & & S'\ar@{^{(}->}[ur]^i &  }$$
(here~$i$ is the shelf inclusion). A shelf morphism $f \colon S_1 \to S_2$ induces a morphism of complexes $f_* \colon C^\bullet(S_2,A) \to C^\bullet(S_1,A), \, \phi \mapsto \phi \circ f^{\times \bullet}$, which
\begin{itemize}
\item restricts to invariant sub-complexes, yielding $f_* \colon C_{inv}^\bullet(S_2,A) \to C_{inv}^\bullet(S_1,A)$, 
\item and induces a morphism in cohomology, $f_*^\# \colon H^\bullet(S_2,A) \to H^\bullet(S_1,A)$.
\end{itemize}
 One thus obtains a commutative diagram of $R$-module morphisms 
$$\xymatrix@!0 @R=1cm @C=3cm{
& H^\bullet(S,A) \ar[dl]_{i_*^\#} & & H^\bullet(S,A) \ar[ll]_{t_*^\# = \Id} \ar[dl]_{i_*^\#} \\
H^\bullet(S',A) & & H^\bullet(S',A) \ar[ul]_{\tau_*^\#} \ar[ll]_{\Id} & }$$
Observe that $t_*^\# = \Id$ since $t \in T_S$ induces the trivial action in cohomology (Lemma~\ref{L:CohomProperties}). Thus~$i_*^\#$ and~$\tau_*^\#$ are mutually inverse $R$-module isomorphisms.

Let us now turn to the commutative diagram of invariant complexes
$$\xymatrix@!0 @R=1cm @C=3cm{
& C_{inv}^\bullet(S,A) \ar[dl]_{i_*} & & C_{inv}^\bullet(S,A) \ar[ll]_{t_* = \Id} \ar[dl]_{i_*} \\
C_{inv}^\bullet(S',A) & & C_{inv}^\bullet(S',A) \ar[ul]_{\tau_*} \ar[ll]_{\Id} & }$$
Here $t_* = \Id$ because of the invariance. Hence the restriction map~$i_*$ gives an isomorphism of complexes $C_{inv}^\bullet(S,A) \overset{\sim}{\to} C_{inv}^\bullet(S',A)$, and~$\tau_*$ is its inverse. 
\end{proof}

In Theorem~\ref{T:CohomComput} we saw that the cohomology of a shelf can often be described in terms of its orbits. It is thus important to understand how a retraction behaves on orbits.

\begin{proposition}\label{P:RetractOrbits}
Let~$S'$ be a retract of a shelf $(S,\op)$. The pre-image of any orbit of~$S'$ is an orbit of~$S$.
\end{proposition}
In particular, a retraction induces a bijection of orbits.

\begin{proof}
Let~$t$ be a retraction for~$S'$. Relation $a = c \op b$ in~$S$ implies $t \op a = (t \op c) \op (t \op b)$ in~$S'$, so~$t$ sends all elements from an orbit of~$S$ to the same orbit of~$S'$. On the other hand, given $a,b \in S$, relation $t \op a {\sim_{S'}} \, t \op b$ implies $a {\sim_{S}}\, t \op a {\sim_{S}}\, t \op b {\sim_{S}}\, b$. Summarizing, one sees that two elements lie in the same orbit of~$S$ if and only if~$t$ sends them to the same orbit of~$S'$.
\end{proof}

Combining Theorems~\ref{T:CohomComput}-\ref{T:RetractHom}, Proposition~\ref{P:RetractOrbits}, and Example~\ref{Ex:Rack}, one gets
\begin{theorem}\label{T:FinShelfHom} 
Let $(S,\op)$ be a quasi-finite shelf, and let~$A$ be an $R$-module. Suppose that the size $|T_{\MRR(S)}|$ of the translation semigroup of the rack type of~$S$ in invertible in~$R$. Then for all $k \ge 0$, one has the following morphism of $R$-modules:
\begin{align*}
H^k(S,A) \cong &\, A^{\Orb(S)^{\times k}}  \cong A^{\Orb(\MRR(S))^{\times k}},\\
[\phi \circ pr^{\times k}] \mapsfrom & \; \phi,
\end{align*}
where the projection $pr \colon S \twoheadrightarrow \Orb(S)$ sends an element of~$S$ to its orbit.
\end{theorem}

In particular, the rational cohomology $\QQ$-modules $H^\bullet(S,\QQ)$ of a quasi-finite shelf are freely generated by the classes of the orbit indicator functions.

\begin{example}\label{Ex:ERetract}
Recall the shelves $E_{n,\rho,\mu}$ and $E^*_{n,\rho,\mu}$ from Example~\ref{Ex:E}. The sub-set 
$$S'=\{(2^n,\rho(2^n)), (2^n,\rho(2^n)+1),\ldots,(2^n,\rho(2^n)+\mu(2^n)-1)\}$$
is a sub-shelf isomorphic to the cyclic rack $C_{0,\mu(2^n)}$ both for~$E$ and~$E^*$. Even better: $S'$ is a rack retract of our shelves, with the retraction  $t = \tau_{(2^n-1,0)}^{\mu(2^n)s}$ for a sufficiently large~$s$. The central property~\eqref{E:LT2} of~$2^n$ implies that~$S'$ is stable by~$T_S$; Theorem~\ref{T:RetractRack} then yields the uniqueness of rack retract for $E^{(*)}_{n,\rho,\mu}$. The rack type of $E^{(*)}_{n,\rho,\mu}$ is thus $C_{0,\mu(2^n)}$. This reduces the cohomology study of the extremely rich family~$E^{(*)}$ to that of cyclic racks. Theorem~\ref{T:FinShelfHom} describes $H^k(E^{(*)}_{n,\rho,\mu},A) \cong A$ as containing the classes of the constant maps only for any module~$A$ over a commutative ring~$R$ containing~$\frac{1}{\mu(2^n)}$. In particular, for the Laver table~$A_n$ this description is valid for any~$R$, since in this case $\mu(2^n)=1$. In fact, in Section~\ref{S:Cyclic} we will show that any ring~$R$ can also be taken for cyclic racks -- and hence for all the~$E^{(*)}$.
\end{example}

\begin{example}\label{Ex:1Retract'}
Recall the shelf $(S,a \op b = a)$ from Example~\ref{Ex:1Retract}. Its translation semigroup is $T_{S}=\{\,\tau_a\,|\,a \in S \,\}$, and any of its elements~$\tau_a$ turns out to be a retraction to the rack retract $\{a\}$. In this case rack retracts are very far from being unique. The rack type of~$S$ is the trivial one-element rack. Theorem~\ref{T:FinShelfHom} then describes $H^k(S,A) \cong A$ as containing the classes of the constant maps only, for any commutative ring~$R$ and any $R$-module~$A$. Note that Theorem~\ref{T:CohomComput} was not sufficient for treating this shelf, since it admits no strong projectors. 
\end{example}

We now explore how rack retracts are related to (semi-)projective families -- and thus to (semi-)strong projectors (Proposition~\ref{P:ProjectorAverage}).

\begin{definition}
An \emph{ideal} of a shelf $(S,\op)$ is its sub-set~$S'$ which is $T_S$-stable, i.e., $a \op S' \subseteq S'$ holds for all $a \in S$. It is called a \emph{rack-ideal} if moreover all $\tau_a$, $a \in S$ induce bijections $S' \to S'$. A \emph{rack-ideal retract} is a rack-ideal which is also a retract.
\end{definition} 

A normal subgroup of a group is an example of a rack-ideal of its conjugation shelf. The rack retracts from Example~\ref{Ex:ERetract} are rack-ideal retracts, while those from Example~\ref{Ex:1Retract'} are not (unless the shelf consists of one element only).

The following elementary properties of (rack-)ideals can be easily verified:

\begin{lemma}\label{L:RackIdeals}
\begin{enumerate}
\item An ideal of a shelf is its sub-shelf, a rack-ideal is a sub-rack, and a rack-ideal retract is a rack retract.
\item The intersection or union of any number of ideals is an ideal.
\item The disjoint union of any number of rack-ideals is a rack-ideal.
\item A rack-ideal is contained in any retract of the shelf.
\end{enumerate}
\end{lemma}

\begin{proposition}\label{P:RetractVsProj}
Let $(S,\op)$ be a quasi-finite shelf. Then it has a semi-projective family. Moreover, the following assertions are equivalent:
\begin{enumerate}
\item $S$ has a projective family;
\item $S$ admits a rack-ideal retract;
\item $S$ admits a unique rack retract.
\end{enumerate}
\end{proposition}

\begin{proof}
Theorem~\ref{T:RetractRack} guarantees the existence of a rack retract~$S'$ of $(S,\op)$, with a retraction~$t$. We now show that the finite family $T_St$ is semi-projective, i.e., $T_St \tau_a = T_St$ for all $a \in S$. Take $u,v \in T_S$ such that $u t \tau_a = v t \tau_a$. This implies $u \tau_{t \op a} t = v \tau_{t \op a} t$, that is,  $u \tau_{t \op a} = v \tau_{t \op a}$ on $t \op S = S'$. Since~$S'$ is a rack, $\tau_{t \op a}$ restricts to a bijection $S' \to S'$. Hence~$u$ and~$v$ coincide on $S' = t \op S$, yielding   
$u t = v t$. Thus $ut \mapsto u t \tau_a$ is an injection $T_St \hookrightarrow T_St \tau_a$, and hence a bijection by finiteness.

Again according to Theorem~\ref{T:RetractRack}, the uniqueness of the rack retract~$S'$ is equivalent to all the $\tau_a$, $a \in S$ inducing rack automorphisms of~$S'$, which is precisely the definition of a rack-ideal retract. In this case the family $T_St$ is projective. Indeed, relation $\tau_a u t = \tau_a v t$ for $u,v \in T_S$, $a \in S$ implies $\tau_a u = \tau_a v$ on $t \op S = S'$, and thus $u=v$ on~$S'$ since $u$, $v$, and $\tau_a$ all restrict to automorphisms of~$S'$; this shows the bijectivity of the maps $ut \mapsto\tau_a u t $ from~$T_St$ to~$\tau_a T_St$.

Now suppose that $T' \subseteq T_S$ is a projective family for~$S$. For every orbit~$\OO$ of~$S$, there exists a finite set~$S_{\OO} \subseteq \OO$ such that, for all $b \in \OO$, one has $T' \op b = S_{\OO}$  (Lemma~\ref{L:proj}). Property $\tau_a T' = T'$ implies that all the $\tau_a$ act on~$S_{\OO}$ by permutations. Thus~$S_{\OO}$ is a rack-ideal, and so is the disjoint union $S' = \bigsqcup_{\OO \in \Orb(S)} S_{\OO} = T' \op S$ (Lemma~\ref{L:RackIdeals}). It remains to show that~$S'$ is a retract. Take a $t \in T'$. As usual, the finiteness of~$T_S$ yields $t^nt^n = t^n$ for some $n \in \NN$. Further, $t^n \in T'$ implies $t^n \op S \subseteq T' \op S = S'$. One the other hand, the rack-ideal~$S'$ satisfies $t^n \op S' = S'$. Thus~$t^n$ is a retraction to $t^n \op S = S'$.
\end{proof}

We finish this section by establishing that a certain reduction procedure for shelves respects retracts. This will be used in Section~\ref{S:Cyclic} to reduce the cohomology study of all finite monogenic shelves to that of cyclic racks.

\begin{definition}
For a shelf $(S,\op)$, consider the equivalence relation
$$ (\, a \approx b \,)\qquad \Longleftrightarrow \qquad (\, \forall c \in S, c \op a = c \op b \,).$$
The quotient $\overline{S} = S / \approx$ is called the \emph{reduced} of $(S,\op)$, and the projection $r:S \twoheadrightarrow \overline{S}$ is called the \emph{reduction} of $(S,\op)$.
\end{definition}

\begin{lemma}
Given a shelf $(S,\op)$, the operation~$\op$ induces a shelf operation on~$\overline{S}$, rendering the reduction~$r$ a shelf morphism.
\end{lemma}

We keep the notation~$\op$ for the induced operation on~$\overline{S}$.

\begin{proposition}\label{P:ReducRetract}
Let~$S'$ be a retract of a shelf $(S,\op)$. Then the reduced~$\overline{S}$ of~$S$ admits a retract isomorphic (as a shelf) to~$S'$. Conversely, any retract of~$\overline{S}$ is isomorphic to a retract of~$S$. 
\end{proposition}

\begin{proof}
Let~$S'$ be a retract of~$S$ with a retraction~$t$, and denote by $\tau \colon S \twoheadrightarrow S'$ the induced projection. The latter factors through~$\overline{S}$, since for all $a \approx b$ one has $t \op a = t \op b$. This yields the commutative diagram of shelf morphisms
$$\xymatrix@!0 @R=1.5cm @C=2.5cm{
 **[r]S'\; \ar@{->>}[d]^{r'} \ar@{^{(}->}[r]^i \ar@/^1.7pc/[rr]^{\Id} & S \ar@{->>}[d]^{r} \ar@{->>}[r]^{\tau} & S' \\
r(S') \ar@{^{(}->}[r]^{\overline{i}} & \overline{S} \ar[ur] & }$$
where~$r'$ is the restriction of the reduction~$r$, and~$i$ and~$\overline{i}$ are shelf inclusions. From this diagram one deduces the injectivity of~$r'$, which is thus a shelf isomorphism. The reduction~$r$ induces an obvious map $T_r \colon T_S \twoheadrightarrow T_{\overline{S}}$ sending a translation~$\tau_a$ to~$\tau_{r(a)}$. Consider now one more commutative diagram of shelf morphisms:
\begin{align}\label{E:DiagReduc}
\xymatrix@!0 @R=1.5cm @C=2.5cm{
 **[r]S'\; \ar@{<->}[d]^{r'} \ar@{^{(}->}[r]^i \ar@/^1.7pc/[rr]^{\Id} & S \ar@{->>}[d]^{r} \ar@{->>}[r]^{\tau} & S' \ar@{<->}[d]^{r'}\\
 **[l]r(S') \ar@{^{(}->}[r]^{\overline{i}} & \overline{S} \ar[r]^{\overline{\tau}} &  **[r]r(S') }
\end{align}
where~$\overline{\tau}$ is induced by~$T_r(t)$. The bijectivity of~$r'$ yields $\overline{\tau} \circ \overline{i} = \Id_{r(S')}$, hence $r(S') \cong S'$ is a retract of~$\overline{S}$, with a retraction~$T_r(t)$.

Now, let~$\overline{S}'$ be a retract of~$\overline{S}$ with a retraction $\overline{t} \in T_{\overline{S}}$, and chose one of the lifts $t \in T_S$ of~$\overline{t}$ with respect to~$T_r$. By the definition of~$T_r$, one has the equalities $r \circ t = \overline{t} \circ r$ and $r \circ t^2 = \overline{t}^2 \circ r$ of maps $S \to \overline{S}$. The four of them thus coincide, since $\overline{t}^2 = \overline{t}$ (Lemma~\ref{L:tt}). In particular, one has $t^2 \op b \approx t \op b$ for all $b \in S$, so, by the definition of the relation~$\approx$, one has $t^3 \op b = t \op (t^2 \op b) = t \op (t \op b) = t^2 \op b$, implying $t^3 = t^2$, and hence $t^4 = t^2$, in~$T_S$. Consequently, $S'=t^2 \op S$ is a retraction of~$S$ (Lemma~\ref{L:tt} again). Further, $r \circ t^2 = \overline{t} \circ r$ implies that~$r$ restricts to a morphism of shelves $r' \colon S' \to \overline{S}'$, surjective since~$r$ is so and since $\overline{t} \op \overline{S} = \overline{S}'$. It remains to show its injectivity. Suppose that~$r'$ coincides on two elements $t^2 \op a$ and $t^2 \op b$ of~$S'$, with $a,b \in S$. It means $t^2 \op a \approx t^2 \op b$, hence $t^2 \op a = t^3 \op a = t^3 \op b = t^2 \op b$, as desired. 
\end{proof} 

Applying the cohomology comparison from Theorem~\ref{T:RetractHom} to the diagram~\eqref{E:DiagReduc}, and recalling the rack retract existence statements from Theorem~\ref{T:RetractRack}, one obtains
\begin{corollary}\label{C:ReducHom}
\begin{enumerate}
\item For a quasi-finite shelf, reduction induces an isomorphism in cohomology.
\item The rack types of a quasi-finite shelf and its reduced are the same.
\end{enumerate}
\end{corollary}

\section{Cohomological aspects of cyclic shelves}\label{S:Cyclic}  

In the remaining sections, we prove Theorem~A by presenting an explicit description of all cocycles, coboundaries, and cohomology groups for a cyclic shelf / Laver table $(S,\op)$. Besides their explicit character, our calculations have the advantage of treating arbitrary coefficients, whereas for cyclic shelves the methods from Sections~\ref{S:Proj} and~\ref{S:Retract} impose some conditions on the coefficient group~$A$. We start with the scheme of our proof, common for the two cases:

\begin{description}
\item[Step 1] Exhibit a family of $k$-tuples $\oa^{(i)}$, $i \in I$, such that any cocycle $\phi \in Z^k(S)$\footnote{From now on we omit the coefficients $A = \ZZ$ for brevity, putting $H^k(S) = H^k(S,\ZZ)$, etc.} is completely determined by its values on these $k$-tuples (in other words, $\phi(\oa^{(i)})$ vanishes for all~$i$ if and only if~$\phi$ is the zero map).
\item[Step 2] Find cocycles $\phi^{(i)} \in Z^k(S)$, $i \in I$, such that the matrix $(\phi^{(i)}(\oa^{(j)}))_{i,j \in I}$ is invertible (i.e., belongs to $GL_{|I|}(\ZZ)$).
\item[Step 3] Show that for some $i_0 \in I$,  the cocycles $\phi^{(i)}$, $i \in I \setminus \{i_0\}$ are coboundaries generating the abelian group~$B^k(S)$.
\end{description} 
If all these properties are established, one immediately concludes that:
\begin{itemize}
\item maps $\phi^{(i)}$, $i \in I \setminus \{i_0\}$ form a basis of~$B^k(S)$;
\item completed with $\phi^{(i_0)}$, they yield a basis of~$Z^k(S)$;
\item $H^k(S) \cong \ZZ$, the class of~$\phi^{(i_0)}$ being its generator.
\end{itemize}
This implies Theorem~A.

\begin{remark}
For expository reasons, we work with coefficients in $A=\ZZ$ only. Our proof extends to the case of a general abelian group~$A$ simply by replacing  all the maps $\phi\colon S^{\times k} \rightarrow \ZZ$ below with families of maps $(\phi_\alpha\colon \oa \mapsto \phi(\oa) \alpha)_{\alpha \in A}$. 
Alternatively, one can use the universal coefficient theorem.
\end{remark}
 
Most maps $\phi\colon S^{\times k} \rightarrow \ZZ$ used in the proof will be expressed in terms of \textit{generalized Kronecker delta functions}
\begin{equation*}
\delta_{\oa}(\ob)= \begin{cases} 1 & \text{ if $a_r=b_r$ for all $r$,} \\ 0 & \text{ otherwise.}\end{cases}
\end{equation*}

In this section we work with a cyclic shelf~$C_{r,m}$, where $r \ge 0, m \ge 1$. In order to simplify the definition~\eqref{E:CShelf} of its shelf operation, we declare $(m-1)+1=0$. Also, for a tuple $\oc =(c_1,\ldots,c_t)$, we put $\oc+1 = (c_1+1,\ldots,c_t+1)$.

We will now prove Theorem~A for~$C_{r,m}$ and~$k>0$, avoiding the trivial case $k=0$.
The proof will rely on a thorough study of tuples containing consecutive elements $a+1,a$ at some place. Concretely, consider the sets
\begin{align*}
I^-_{r,m,k} &= \big\{\, \oa \in C_{r,m}^{\times k} \,\big|\, a_1 = a_3 = \ldots = a_{2t-1}=0, a_2 = a_4 = \ldots = a_{2t}=m-1, \\
& \qquad\qquad\qquad a_{2t+1}=a_{2t+2}+1, a_{2t+2} \neq m-1 \text{ for some } 0 \le t \le \tfrac{k}{2} -1\,\big\},\\ 
I_{r,m,k} &= I^-_{r,m,k} \coprod \big\{\, \oi_0 = (0,m-1,0,m-1,\ldots)  \in C_{r,m}^{\times k} \, \big\}.
\end{align*}

A straightforward computation of the size of these sets allows us to recover the polynomials~\eqref{E:Pk} announced by Theorem~A: 

\begin{lemma}
One has $|I^-_{r,m,k}| = P_k(r+m)$.
\end{lemma}

Step 1 from our general proof scheme follows from
\begin{lemma}\label{L:CyclicA}
A cocycle $\phi \in Z^k(C_{r,m})$ vanishing on all $\oa \in I_{r,m,k}$ is the zero cocycle.
\end{lemma} 

\begin{proof}
Use induction on~$k$. Case $k=1$ is easy. Suppose now $k \ge 2$.
For any $\oc \in C_{r,m}^{\times (k-1)}$ and $a \in C_{r,m}$, $a \neq m-1$, one has
\begin{align*}
0 &= (d^k\phi)(a+1,a,\oc) \\
  & = (\phi(a+1,\oc+1)-\phi(a,\oc))-(\phi(a+1,\oc+1)-\phi(a+1,\oc))\\
  & = \phi(a+1,\oc)-\phi(a,\oc)
\end{align*}
(we omitted further terms: they have the form $\phi(a+1,a,\ldots)$, and vanish since $(a+1,a,\ldots)$ lies in~$I_{r,m,k}$), implying
\begin{align}\label{E:FirstTerm}
\phi(-r,\oc)  & = \phi(-(r-1),\oc) = \ldots = \phi(m-1,\oc).
\end{align}
In particular, for any $\oc \in C_{r,m}^{\times (k-2)}$, $a,b \in C_{r,m}$, $b \neq m-1$, one obtains
\begin{align}
\phi(a,b,\oc)  & = \phi(b+1,b,\oc) = 0,\label{E:FirstTerm2}\\
\phi(a,m-1,\oc)  & = \phi(0,m-1,\oc).\label{E:FirstTerm3}
\end{align}
For $k=2$, one concludes by recalling that $(0,m-1) \in I_{r,m,2}$. For $k>2$, define $\ophi \in C^{k-2}(C_{r,m})$ by $\ophi(\oc)=\phi(0,m-1,\oc)$. Now, for any $\oc \in C_{r,m}^{\times (k-1)}$, one calculates
\begin{align*}
0 &= (d^k\phi)(0,m-1,\oc) \\
  & = (\phi(0,\oc+1)-\phi(m-1,\oc))-(\phi(0,\oc+1)-\phi(0,\oc))+ (d^{k-2}\ophi)(\oc)\\
  & = (d^{k-2}\ophi)(\oc)
\end{align*}
(in the last step we used~\eqref{E:FirstTerm}). Hence~$\ophi$ is a $k-2$ cocycle. It vanishes on all $\oa \in I_{r,m,k-2}$, since $(0,m-1,\oa) \in I_{r,m,k}$ for such~$\oa$. The induction hypothesis implies that~$\ophi$ is the zero map, hence, according to~\eqref{E:FirstTerm2}-\eqref{E:FirstTerm3}, so is~$\phi$.
\end{proof}

We then construct a basis for cocycles. Take a $k$-tuple $\oa \in I^-_{r,m,k}$, presented as
 $((0,m-1)^t,a_{2t+2}+1, a_{2t+2},\oc)$, where $a_{2t+2} \neq m-1$, and~$(0,m-1)^t$ denotes the sequence~$(0,m-1)$ repeated~$t$ times. Associate to it the $k-1-2t$-tuple $\ooa = (a'_{2t+2},\oc)$, where $a'_{2t+2}=a_{2t+2}$ except when~$\oc$ is non-empty and satisfies $a_{2t+2} = c_1+1$, in which case put $a'_{2t+2}=m-1$. Next, define the following map in~$C^k(C_{r,m})$:
\begin{align*}
\phi^{(\oa)}(b_1,\ldots,b_k) &= \delta_{m-1}(b_2)\delta_{m-1}(b_4)\cdots\delta_{m-1}(b_{2t}) d^{k-1-2t} \delta_{\ooa}(b_{2t+1},\ldots,b_k).
\end{align*}

For $k=1$, take as~$\phi^{(\oi_0)}$ the constant map $b \mapsto 1$. For a~$k \ge 2$, $k\in\{2s,2s+1\}$, put
\begin{align*}
\phi^{(\oi_0)}(b_1,\ldots,b_k) &= \delta_{m-1}(b_2)\delta_{m-1}(b_4)\cdots\delta_{m-1}(b_{2s}).
\end{align*}

\begin{lemma}
The map~$\phi^{(\oi_0)}$ is a cocycle, and all the maps~$\phi^{(\oa)}$ are coboundaries. 
\end{lemma}

\begin{proof}
Regroup the terms in the definition~\eqref{E:RackCohom} of our differential as follows:
\begin{align}
& \qquad\qquad\qquad (d^k \phi)(b_1, \ldots, b_{k+1}) =\notag\\
 \sum_{i=1}^{s} ((& \phi(b_1,\ldots,b_{2i-2},{b_{2i}+1}, b_{2i+1}+1, \ldots)-\phi(b_1,\ldots,b_{2i-2},{b_{2i-1}},  b_{2i+1}+1, \ldots))\notag\\ 
& - (\phi(b_1,\ldots,b_{2i-2},{b_{2i}},b_{2i+1}, \ldots)- \phi(b_1,\ldots,b_{2i-2},{b_{2i-1}},b_{2i+1}, \ldots))),\label{E:CocRegrouped}
\end{align}
where $k+1 \in \{2s,2s+1\}$. Now, the value of~$\phi^{(\oi_0)}$ on a $k$-tuple does not depend on the odd coordinates of the tuple, implying
\begin{align*}
\phi^{(\oi_0)}(b_1,\ldots,b_{2i-2},\underline{b_{2i}+1}, b_{2i+1}+1, \ldots) &=\phi^{(\oi_0)}(b_1,\ldots,b_{2i-2},\underline{b_{2i-1}},  b_{2i+1}+1, \ldots),\\
\phi^{(\oi_0)}(b_1,\ldots,b_{2i-2},\underline{b_{2i}},b_{2i+1}, \ldots) &= \phi^{(\oi_0)}(b_1,\ldots,b_{2i-2},\underline{b_{2i-1}},b_{2i+1}, \ldots).
\end{align*}
Hence, according to~\eqref{E:CocRegrouped}, $d^k \phi^{(\oi_0)}$ is zero, as desired.

For an $\oa = ((0,m-1)^t,a_{2t+2}+1, a_{2t+2},\oc) \in I^-_{r,m,k}$ (with $a_{2t+2} \neq m-1)$), we will prove that~$\phi^{(\oa)}$ is a coboundary by showing that
\begin{align*}
\phi^{(\oa)} &= d^{k-1}\psi, & \psi &=  \sum_{u_1,\ldots,u_t \in C_{r,m}} \delta_{u_1,m-1,\ldots,u_t,m-1,\ooa}.
\end{align*}
The value of~$\psi$ on a $(k-1)$-tuple does not depend on the coordinates $1,3,\ldots,2t-1$ of the tuple. Thus \eqref{E:CocRegrouped} gives
\begin{align*}
& \qquad\qquad\qquad (d^{k-1} \psi)(b_1, \ldots, b_{k}) =\\
 \sum_{i=t+1}^{s} ((& \psi(b_1,\ldots,b_{2i-2},{b_{2i}+1}, b_{2i+1}+1, \ldots)-\psi(b_1,\ldots,b_{2i-2},{b_{2i-1}},  b_{2i+1}+1, \ldots))\\ 
& - (\psi(b_1,\ldots,b_{2i-2},{b_{2i}},b_{2i+1}, \ldots)- \psi(b_1,\ldots,b_{2i-2},{b_{2i-1}},b_{2i+1}, \ldots))),
\end{align*}
where $k \in \{2s,2s+1\}$. Recalling the definition of~$\psi$, one simplifies this as 
\begin{align}
& \qquad\qquad (d^{k-1} \psi)(b_1, \ldots, b_{k}) =\delta_{m-1}(b_2)\delta_{m-1}(b_4)\cdots\delta_{m-1}(b_{2t})\times \label{E:EvalBasis}\\
&\sum_{i=t+1}^{s} (( \delta_{\ooa}(b_{2t+1},\ldots,b_{2i-2},{b_{2i}+1}, b_{2i+1}+1, \ldots)-\delta_{\ooa}(b_{2t+1},\ldots,b_{2i-2},{b_{2i-1}},  b_{2i+1}+1, \ldots))\notag\\ 
&\qquad - (\delta_{\ooa}(b_{2t+1},\ldots,b_{2i-2},{b_{2i}},b_{2i+1}, \ldots)- \delta_{\ooa}(b_{2t+1},\ldots,b_{2i-2},{b_{2i-1}},b_{2i+1}, \ldots)))\notag\\
&\qquad= \delta_{m-1}(b_2)\delta_{m-1}(b_4)\cdots\delta_{m-1}(b_{2t}) d^{k-1-2t} \delta_{\ooa}(b_{2t+1},\ldots,b_k)\notag\\
&\qquad= \phi^{(\oa)}(b_1, \ldots, b_{k}).\notag\qedhere
\end{align}\qedhere
\end{proof}

We are now ready to complete Step 2 of our general proof scheme.
\begin{lemma}
Matrix $M=(\phi^{(\oi)}(\oj))_{\oi, \oj \in I_{r,m,k}}$ is invertible.
\end{lemma}

\begin{proof}
First, divide the set~$I_{r,m,k}$, with $k \in \{2s,2s+1\}$, into sub-sets $B_s=\{\oi_0\}$ and $B_t$, $0 \le t < s$ , comprising elements $\oa = ((0,m-1)^t,a_{2t+2}+1, a_{2t+2},\oc)$ with $a_{2t+2} \neq m-1$. The latter condition rewrites as $\delta_{m-1}(a_{2t+2}) = 0$. Therefore, $\phi^{(\oi)}(\oj)=0$ if $\oi \in B_u$, $\oj \in B_t$, and $u > t$. Matrix~$M$ is thus block upper triangular. Its last block is the $1 \times 1$ matrix containing~$1$. Let us compute its blocks~$(t,t)$ for $0 \le t \le s-1$. Take  $\oj = ((0,m-1)^t,j_{2t+2}+1, j_{2t+2},\oc) \in B_t$, and $\oi \in B_t$, with $\ooi =(i'_{2t+2},\od)$ and $i'_{2t+2} \neq d_1+1$ (or $i'_{2t+2} \neq m-1$ if~$\od$ is empty). Splitting formula~\eqref{E:EvalBasis} into two parts, one gets
\begin{align*}
\phi^{(\oi)}(\oj) =& ( \delta_{\ooi}(j_{2t+2}+1,\oc+1)-\delta_{\ooi}(j_{2t+2}+1,\oc+1)) - (\delta_{\ooi}(j_{2t+2},\oc)-\delta_{\ooi}(j_{2t+2}+1,\oc)) \\
&+\delta_{(i'_{2t+2},d_1)}(j_{2t+2}+1, j_{2t+2}) d^{k-2t-3} \delta_{d_2,\ldots}(\oc)
\end{align*}
(with the last term omitted if~$\od$ is empty). The last term vanishes, since $i'_{2t+2} \neq d_1+1$. Our evaluation thus becomes very simple:
\begin{align*}
\phi^{(\oi)}(\oj) =& \delta_{\ooi}(j_{2t+2}+1,\oc) - \delta_{\ooi}(j_{2t+2},\oc)= (\delta_{i'_{2t+2}}(j_{2t+2}+1) - \delta_{i'_{2t+2}}(j_{2t+2}))\delta_{\od}(\oc).
\end{align*}
In particular, one observes that the block~$(t,t)$ is in its turn a block diagonal matrix, with elements $((0,m-1)^t,j_{2t+2}+1, j_{2t+2},\oc)$ sharing the same~$\oc$ regrouped together. Order such elements according to their component~$j_{2t+2}$. With this ordering, the sub-block corresponding to a $\oc \in C_{r,m}^{\times (k-2t-2)}$ can be obtained from the matrix 
$$(\delta_{i}(j+1) - \delta_{i}(j))_{i,j \in C_{r,m}, j \neq m-1} = 
\begin{pmatrix}
  -1 & 0 & 0 & \cdots & 0 & 0 \\
  1 & -1 & 0 & \cdots & 0 & 0 \\
  0 & 1 & -1 & \cdots & 0 & 0 \\  
  \vdots & \vdots & \vdots & \vdots  & \ddots & \vdots  \\
  0 & 0 & 0 & \cdots & -1 & 0 \\ 
  0 & 0 & 0 & \cdots & 1 & -1 \\ 
  0 & 0 & 0 & \cdots & 0 & 1 \\  
 \end{pmatrix}.$$ 
 by removing line~$c_1+1$ or~$m-1$ (depending on the condition on~$i'_{2t+2}$). Each sub-block is thus invertible, hence so is the whole matrix~$M$.
\end{proof}

The remaining Step 3 follows from

\begin{lemma}
Coboundaries $\phi^{(\oa)}$, $\oa \in I^-_{r,m,k}$, generate~$B^k(C_{r,m})$.
\end{lemma}

\begin{proof}
Any coboundary is a cocycle, and thus, according to previous lemmas, can be presented as a linear combination of coboundaries $\phi^{(\oa)}$, $\oa \in I^-_{r,m,k}$ and of~$\phi^{(\oi_0)}$. In order to see that the latter has coefficient~$0$ in this linear combination, remark that 
\begin{gather*}
\textstyle \sum_{\ob \in D_{m,k}}\phi^{(\oi_0)}(\ob) = \phi^{(\oi_0)}(\oi_0)=1,\\
D_{m,k} = \big\{\, \ob \in \{0,1,\ldots, m-1\}^k \,\big|\, b_1 = b_2+1, \ldots, b_{2s-1} = b_{2s}+1,b_{2s+1} =0 \,\big\}
\end{gather*}
(here as usual $k \in \{2s,2s+1\}$, and condition $b_{2s+1} =0$ is omitted if $k=2s$). We will now prove  $\sum_{\ob \in D_{m,k}}\phi(\ob) =0$ for any coboundary $\phi = d^{k-1}\psi$. For a~$\ob \in D_{m,k}$, regroup the terms of $d^{k-1}\psi(\ob)$ as in~\eqref{E:CocRegrouped}. Condition $b_{2i-1} = b_{2i}+1$ implies 
\begin{align*}
\psi(b_1,\ldots,b_{2i-2},\underline{b_{2i}+1}, b_{2i+1}+1, \ldots) &=\psi(b_1,\ldots,b_{2i-2},\underline{b_{2i-1}},  b_{2i+1}+1, \ldots),
\end{align*}
killing a half of the terms summed in~\eqref{E:CocRegrouped}. The sum $\sum_{\ob \in D_{m,k}}(d^{k-1} \psi)(\ob)$ thus becomes
\begin{align*}
-&\sum_{i=1}^{s}  \sum_{\ob \in D_{m,k}}(\phi(b_1,\ldots,b_{2i-2},{b_{2i}},b_{2i+1}, \ldots)- \phi(b_1,\ldots,b_{2i-2},{b_{2i-1}},b_{2i+1}, \ldots))=\\
 -&\sum_{i=1}^{s}  \sum_{\quad \substack{(b_1,\ldots,b_{2i-2},b_{2i+1},\ldots,b_k) \\ \in D_{m,k-2}}\quad} \sum_{b_{2i}=0}^{m-1}(\phi(b_1,\ldots,b_{2i-2},\underline{b_{2i}},b_{2i+1}, \ldots) \\
&\qquad\qquad\qquad\qquad\qquad \qquad\qquad\qquad -  \phi(b_1,\ldots,b_{2i-2},\underline{b_{2i}+1},b_{2i+1}, \ldots)),
\end{align*}
which vanishes as announced.
\end{proof}

All the steps from our proof scheme are now completed. Put together, they give

\begin{theorem}\label{T:Cyclic}
Take $r,k \ge 0$, $m \ge 1$. Write~$k$ as~$2s$ or~$2s+1$.
\begin{enumerate}
\item The $k$-coboundaries of~$C_{r,m}$ form a free $\ZZ$-module~$B^k(C_{r,m})$ of rang~$P_k(r+m)$, with a basis given by the maps 
\begin{align*}
\phi^{(\oa)}(b_1,\ldots,b_k) &= \delta_{m-1}(b_2)\delta_{m-1}(b_4)\cdots\delta_{m-1}(b_{2t}) d^{k-1-2t} \delta_{\oa}(b_{2t+1},\ldots,b_k),
\end{align*}
where $0 \le t \le \tfrac{k}{2} -1$, and $\oa \in C_{r,m}^{\times(k-2t-1)}$ satisfies $a_1 \neq a_2+1$ if $k-2t-1 \ge 2$, and $a_1 \neq m-1$ if $k-2t-1 = 1$.
\item For $k \ge 2$, $H^k(C_{r,m})$ is freely generated by the class of the map 
\begin{align*}
\phi^{0}(b_1,\ldots,b_k) &= \delta_{m-1}(b_2)\delta_{m-1}(b_4)\cdots\delta_{m-1}(b_{2s}).
\end{align*}
For $k \le 1$, the class of $1 \in \ZZ$ or of the constant map $b \mapsto 1$ is a free generator.
\item The $k$-cocycles of~$C_{r,m}$ form a free $\ZZ$-module~$Z^k(C_{r,m})$ of rang~$P_k(r+m)+1$, with a basis obtained by adding~$\phi^0$ to our basis for~$B^k(C_{r,m})$.
\end{enumerate}
\end{theorem}

In Section~\ref{S:Proj}, the rational cohomology groups of cyclic shelves $H^k(C_{r,m},\QQ)$ were shown to be freely generated by the class of the constant map $\phi_{const}\colon\ob \mapsto 1$. We now explore its analogue $[\phi_{const}]$ in $H^k(C_{r,m},\ZZ)$.

\begin{proposition}
Take an $m \ge 0$ and a $k \ge 0$ written as~$2s$ or $2s+1$. In $H^k(C_{0,m},\ZZ)$, one has $[\phi_{const}] = m^s [\phi^{0}]$.
\end{proposition}

\begin{proof}
According to Theorem~\ref{T:Cyclic}, the cocycle~$\phi_{const}$ can be presented as $\alpha\phi^{0} + d^{k-1}\psi$ for some $\alpha \in \ZZ$ and $\psi \in C^{k-1}(C_{0,m},\ZZ)$. To determine~$\alpha$, we sum the evaluations of both sides on all $\ob \in C_{0,m}^{\times k}$:  
\begin{align*}
\sum_{\ob \in C_{0,m}^{\times k}} \phi_{const}(\ob) = & \, m^k, \qquad\qquad \sum_{\ob \in C_{0,m}^{\times k}} \phi^{0}(\ob) = m^{k-s},\\
\sum_{\ob \in C_{0,m}^{\times k}} d^{k-1}\psi(\ob) = &\sum_{i=1}^{k} (-1)^{i-1} ( \sum_{\ob \in C_{0,m}^{\times k}} \psi(b_1,\ldots, b_{i-1},b_{i+1}+1, \ldots, b_{k}+1)\\
&\qquad\qquad - \sum_{\ob \in C_{0,m}^{\times k}} \psi(b_1,\ldots, b_{i-1}, b_{i+1}, \ldots, b_{k})) \\
= &\sum_{i=1}^{k} (-1)^{i-1} ( m\sum_{\ob \in C_{0,m}^{\times (k-1)}} \psi(\ob) - m\sum_{\ob \in C_{0,m}^{\times (k-1)}} \psi(\ob))= 0.
\end{align*}
Thus $\alpha = m^s$.
\end{proof}

According to Theorem~\ref{T:Cyclic}, $[\phi^{0}]$ freely generates our integral cohomology group, so, for $m,k \ge 2$, $[\phi_{const}]$ is no longer a generator; it can even vanish in positive characteristic (for instance, for $A = \ZZ_p$). This also shows the existence of a semi-strong projector to be essential for Proposition~\ref{P:CohomDecompose} to hold true. Indeed, $H^2_{inv}(C_{0,m},\ZZ)$ is generated by $[\phi_{const}]$ (Remark~\ref{R:SmallGroups}), and the assertions of Proposition~\ref{P:CohomDecompose} would force it to coincide with the generator $[\phi^{0}]$ of $H^2(C_{0,m},\ZZ)$, which we have just shown to be false.

We will now deduce from Theorem~\ref{T:Cyclic} a description of the cohomology of an arbitrary finite monogenic shelf $(S,\op)$ with any coefficients. According to Dr\'{a}pal \cite{DraAlg,DraGro}, iterated reductions of $(S,\op)$ yield a shelf of type $E_{n,\rho,\mu}$ (Example~\ref{Ex:E}), which admits a cyclic rack~$C_{0,m}$ as a retract (Example~\ref{Ex:ERetract}). One obtains a shelf morphism $pr_S \colon S \twoheadrightarrow C_{0,m}$, which induces an isomorphism in cohomology (Corollary~\ref{C:ReducHom} and Theorem~\ref{T:RetractHom}). Corollary~\ref{C:ReducHom} also identifies the cyclic racks as the only possible rack types for finite monogenic shelves. The cohomology of~$C_{0,m}$ being completely described in Theorem~\ref{T:Cyclic}, one obtains

\begin{theorem}\label{T:FinMonogHom}
For any finite monogenic shelf $(S,\op)$, the $k$th rack cohomology of~$S$ with coefficients in any abelian group~$A$ is described for any $k \ge 0$ by
$$H^k(S,A) \cong A.$$
\end{theorem}

For $A=\QQ$, the cohomology $\QQ$-modules in every degree are freely generated by the classes of the $1$-valued constant maps. 
For a general~$A$, $H^k(S,A)$ can be presented as $\{\, [\phi^{m,k}_\alpha \circ pr_S^{\times k}] \,|\, \alpha \in A \,\}$, where\footnote{For simplicity, we replaced the conditions $b_{2i} = m-1$ from Theorem~\ref{T:Cyclic} by simpler ones $b_{2i} = 0$, which is possible by symmetry considerations.}
$$\phi^{m,k}_\alpha \colon C_{0,m}^{\times k} \to A, \qquad \ob \mapsto \begin{cases} \alpha &\text{ if } b_2 = b_4 = b_6 = \ldots = 0, \\ 0 &\text{ otherwise.} \end{cases}$$

\medskip%
\section{Cohomological aspects of Laver tables}\label{S:Laver}

We now turn to a study of the $k$th cohomology of the Laver table~$A_n$. Here we only indicate how the steps of our general proof scheme are realized in this case, without giving the technical details.

Suppose $k>0$ and $n>0$ (this excludes some cases where Theorem~A is trivial). Our proof heavily uses the properties \eqref{E:LT1}-\eqref{E:LT3} of the elements~$2^n$ and~$N=2^n-1$ of~$A_n$. To exploit the particularities of these elements, put
$$\ob^{(k,r)}=(\underbrace{N,\ldots,N}_{r},\underbrace{2^n,\ldots,2^n}_{k-r}).$$
The evaluation of coboundaries on such $k$-tuples is particularly easy:

\begin{lemma}\label{L:br}
Take a map $\psi\colon A_n^{\times (k-1)} \rightarrow \ZZ$, and an integer~$r$, $0 \le r \le k$, written as~$2t$ or~$2t+1$. One has
\begin{align*}
(d^{k-1}\psi)(\ob^{(k,r)}) &=  \psi(\ob^{(k-1,0)})-\psi(\ob^{(k-1,1)})+\psi(\ob^{(k-1,2)})- \cdots - \psi(\ob^{(k-1,2t-1)}).
\end{align*}
\end{lemma} 

Consider now the sets
\begin{align*}
I^-_{n,k} &= \big\{\, \oa \in A_n^{\times k} \,\big|\, a_1 = \ldots = a_{2t+1}=N, a_{2t+2} \neq N \text{ for some } 0 \le t \le \tfrac{k}{2} -1\,\big\},\\ 
I_{n,k} &= I^-_{n,k} \coprod \big\{\oi_0 = \ob^{(k,k)}  \big\}.
\end{align*}

\begin{lemma}
One has $|I^-_{n,k}| = P_k(2^n)$.
\end{lemma}

We now attack Step 1 from our general proof scheme.
\begin{lemma}
A cocycle $\phi \in Z^k(A_n)$ vanishing on all the $\oa \in I_{n,k}$ is necessarily the zero cocycle.
\end{lemma}

We then construct $|I_{n,k}|$ cocycles necessary for Step 2. Any~$\oa \in I^-_{n,k}$ starts with~$N$; remove this first element~$N$, and denote by~$\ooa$ the $(k-1)$-tuple obtained. Put 
\begin{align*}
\phi^{(\oa)} &= d^{k-1} \delta_{\ooa} \in B^k(S) \subseteq Z^k(S).
\end{align*}
Further, let~$\phi^{(\oi_0)}$ be the constant map $\oc \mapsto 1$, which is obviously a cocycle.

\begin{lemma}
Matrix $M=(\phi^{(\oi)}(\oj))_{\oi,\oj \in I_{n,k}}$ is invertible.
\end{lemma}

Cocycles $\phi^{(\oa)}$, $\oa \in I^-_{n,k}$, are coboundaries by construction. To finish Step 3, one needs the following

\begin{lemma}
Coboundaries $\phi^{(\oa)}$, $\oa \in I^-_{n,k}$, generate~$B^k(S)$.
\end{lemma}

All the steps from our proof scheme are now completed. Put together, they give

\begin{theorem}\label{T:Laver}
Take $n,k \ge 0$.
\begin{enumerate}
\item The $k$-coboundaries of~$A_n$ form a free $\ZZ$-module~$B^k(A_n)$ of rang~$P_k(2^n)$, with a basis given by the maps $d^{k-1} \delta_{\oa}$ for all~$\oa \in A_n^{\times(k-1)}$ satisfying
$a_1 = \ldots = a_{2t}=2^n-1$, $a_{2t+1} \neq 2^n-1$ for some $0 \le t \le \tfrac{k}{2} -1$.
\item $H^k(A_n)$ is freely generated by the class of the constant map $\phi^0\colon\ob \mapsto 1$.
\item The $k$-cocycles of~$A_n$ form a free $\ZZ$-module~$Z^k(A_n)$ of rang~$P_k(2^n)+1$, with a basis obtained by adding~$\phi^0$ to our basis for~$B^k(A_n)$.
\end{enumerate}
\end{theorem}

\bibliographystyle{alpha}
\bibliography{biblio}
\end{document}